\documentclass[a4paper]{amsart}

\usepackage{xcolor}

\newcommand{\U}{\mathcal{U}}

\newcommand{\A}{\mathcal{A}}

\newcommand{\D}{\mathcal{D}}
\newcommand{\M}{\mathcal{M}}
\newcommand{\N}{\mathcal{N}}

\renewcommand{\phi}{\varphi}
\renewcommand{\epsilon}{\varepsilon}

\newcommand{\LL}{\mathcal{L}}
\newcommand{\fin}{\textrm{fin}}
\renewcommand{\P}{\mathcal{P}}

\newcommand{\KK}{ \mathcal{K} }
\newcommand{\NN}{\mathbb{N}}

\newcommand{\ZZ}{\mathbb{Z}}

\newcommand{\TT}{\mathbb{T}}

\newcommand{\img}{\mathrm{img}}
\newcommand{\sym}{\mathrm{Sym}}

\newcommand{\Th}{\textrm{Th}}
\newcommand{\sig}{\mathrm{sig}}
\newcommand{\qtp}{\mathrm{qftp}}

\renewcommand{\models}{\vDash}

\renewcommand{\aa}{\overline{a}}
\newcommand{\bb}{\overline{b}}
\newcommand{\cc}{\overline{c}}

\newcommand{\xx}{\overline{x}}
\newcommand{\yy}{\overline{y}}

\newcommand{\gen}[2]{\left\langle #1 \right\rangle_{#2}}

\renewcommand{\r}{  {\upharpoonright} }

\renewcommand{\iff}{\Leftrightarrow}

\newcommand{\fraisse}{\textrm{Fra\"iss\'e }}
\newcommand{\Fraisse}{\textrm{Fra\"iss\'e }}

\usepackage{amscd}
\DeclareMathOperator{\acl}{acl}

\DeclareMathOperator{\tp}{tp}

\DeclareMathOperator{\qftp}{qftp}

\usepackage{ stmaryrd }
\usepackage{enumerate}
\usepackage{tikz-cd}

\usepackage{hyperref}
\usepackage{fancyhdr}
\usepackage[english]{babel}
\usepackage[T1]{fontenc}
\usepackage[utf8]{inputenc}

\usepackage{amsmath, amsthm, amscd, amsfonts}
\usepackage{amssymb, tikz}
\usepackage{mathtools}
\usepackage{tikz-cd}
\usepackage{amscd}
\usepackage{MnSymbol,bbding,pifont}

\usepackage{hyperref}
\usepackage[capitalise]{cleveref}

\usepackage{setspace}
\makeatletter
\renewcommand\footnotesize{%
	\@setfontsize\footnotesize\@ixpt{9}%
	\abovedisplayskip 8\p@ \@plus2\p@ \@minus4\p@
	\abovedisplayshortskip \z@ \@plus\p@
	\belowdisplayshortskip 4\p@ \@plus2\p@ \@minus2\p@
	\def\@listi{\leftmargin\leftmargini
		\topsep 4\p@ \@plus2\p@ \@minus2\p@
		\parsep 2\p@ \@plus\p@ \@minus\p@
		\itemsep \parsep}%
	\belowdisplayskip \abovedisplayskip
}
\makeatother

\renewcommand{\qtp}{\mathrm{qftp}}
\newcommand{\aut}{\mathrm{Aut}}

\renewcommand{\KK}{\mathbf{K}}
\newcommand{\age}{\mathsf{age}}

\newcommand{\h}{\mathsf{h}}

\renewcommand{\emptyset}{\varnothing}

\newcommand{\Bvert}{\,\Big\vert\,}

\renewcommand{\rho}{\varrho}
\newcommand{\concat}{{^\smallfrown}}

\newcommand{\pred}{\mathtt{pred}}

\renewcommand{\TT}{\mathbf{T}}
\newcommand{\tcl}{\mathrm{tcl}}
\newcommand{\mult}{\mathrm{mult}}
\newcommand{\qr}{\mathrm{qr}}

\newtheorem{thm}{Theorem}[subsection]
\newtheorem{cor}[thm]{Corollary}
\newtheorem{lemma}[thm]{Lemma}
\newtheorem{prop}[thm]{Proposition}

\theoremstyle{remark}
\newtheorem*{claim}{Claim}
\newtheorem*{conv}{Convention}
\newtheorem{rem}[thm]{Remark}
\newtheorem{fact}[thm]{Fact}
\newtheorem{obs}[thm]{Observation}

\newtheorem{notation}[thm]{Notation}

\renewcommand{\r}{  {\upharpoonright} }

\renewcommand{\iff}{\Leftrightarrow}

\theoremstyle{definition}
\newtheorem{defn}[thm]{Definition}

\newcommand{\indep}{\,\,\raise.2em\hbox{$\,\mathrel|\kern-.93em\lower.4em\hbox{$\smile$}$}}
\newcommand{\nindep}{\,\,\raise.2em\hbox{$\mathrel|\kern-.945em\lower.4em\hbox{$\smile$}
		\kern-.75em\hbox{\char'57}$}\;}

\newcommand{\shiftleft}[2]{\makebox[0pt][r]{\makebox[#1][l]{#2}}}

\newcommand{\myind}[1]{   \indep{\raise 7pt\hbox{\shiftleft{4pt}{{\tiny #1}}}}   }
\newcommand{\nmyind}[1]{   \indep{\raise 7pt\hbox{\shiftleft{4pt}{{\tiny #1}}}}{\raise .5pt\hbox{\shiftleft{10pt}{$\not$}}}   }
\newcommand{\myhatind}[1]{  
	verset{\sim}{\indep}{\raise 7pt\hbox{\shiftleft{5pt}{{\tiny #1}}}}  }

\newtheoremstyle{InProofNum}
{\topsep}{\topsep}              
{}                      
{}                              
{\itshape}                     
{.}                             
{ }                             
{\thmname{#1}\thmnote{ #3}}
\theoremstyle{InProofNum}

\begin{document}

	\title{Asymptotic Classes of Trees and $\aleph_0$-categoricity}

\author{Mostafa Mirabi}

\address{Department of Mathematics and Computer Science, The Taft School, 110 Woodbury Rd., Watertown, CT 06795}

\email{mmirabi@wesleyan.edu}

\urladdr{https://sites.google.com/site/mostafamirabi/}


\date{}


\maketitle

\begin{abstract}
This paper focuses on the characterization of $\aleph_0$-categorical theories of trees in the following sense:  for any $\aleph_0$-cateorical theory $T$ of trees there is a tree plan $\Gamma$ (see Definition \ref{def_tree-plan_I(S)}) such that $T=Th(\Gamma(\omega))$ where $\Gamma(\omega)$ is the generic model of a {\Fraisse} class $\KK(\Gamma)$ obtained from the tree plan $\Gamma$. Also, it is shown that $\KK(\Gamma)$ forms an asymptotic class, and its model-theoretic properties have been studied. Moreover, it is demonstrated that 
 asymptotic classes of finite trees yield $\aleph_0$-categorical ultraproducts, and a characterization of 
 supersimple, finite rank trees, using the notion of tree plan, is provided. 
\end{abstract}


\section{Introduction}

The notion of a one-dimensional asymptotic class was introduced by Macpherson and Steinhorn in \cite{macpherson-steinhorn-2007},  in which, roughly speaking, the cardinalities of definable sets are well controlled. Elwes, in \cite{elwes-2007},  generalized this notion and defined an N-dimensional asymptotic class for natural numbers $N\geq 1$.  The motivating example is the class of finite fields. The definition was inspired by a theorem of Chatzidakis, van den Dries and Macintyre, see \cite{c-vdd-m-1992}, which is essentially  a generalization of the classical Lang–Weil estimates for varieties in the class of finite fields. Other examples of asymptotic classes include the class of all finite cyclic groups, the class of Paley graphs, classes of classical geometries (e.g., linear, affine, and projective spaces over finite fields), and any class of finite Moufang polygons \cite{DELLO-STRITTO}. Also, in \cite{Garcia_O-Asymptotic}, Garcia has defined a concept called o-asymptotic class which is a natural generalization of the notion of asymptotic classes to finite ordered structures. 

The study of trees from various perspectives has been a subject of interest in model theory.  In this paper, we study the model theory of asymptotic classes of trees, and provide a  characterization of $\aleph_0$-categorical theories of trees using the notion of tree plans, and also explore a characterization of asymptotic classes of trees. Here by a tree we mean a partially ordered set in a language containing an ordering, a predecessor function, a meet function and a constant symbol that denotes the root of the tree.  The notion of tree plan plays a key role (see Definition \ref{def_tree-plan_I(S)}). Roughly speaking a tree plan is a finite tree together with a function that assigns 1 or $\infty$ to each node, representing the multiplicity of elements in the generic model.

In \cite{MS-meas-Mirabi}, the author, inspired by Hill's results in \cite{hill-coordinatization},  has used the concept of a tree plan to conduct an  analysis of MS-measurable structures. Through this investigation, he made a discovery by establishing a  relationship between coordinatized structures and $\aleph_0$-categorical MS-measurable structures. In particular, he defined a notion of coordinatization for $\aleph_0$-categorical structures which is, like Lie coordinatized structures in \cite{FSWFT}, and demonstrated that a structure coordinatized, in a strong sense, by $\aleph_0$-categorical MS-measurable structures is itself MS-measurable (also see \cite{Mirabi-thesis}).

 The paper is organized as follows. Section \ref{prelim} is devoted for preliminaries including trees, \fraisse limits, and asymptotic classes.  Section \ref{sec-treeplan}  introduces the notion of tree plan and tree-closure which play crucial in the rest of the paper. In Section \ref{K(Gamma) is a Fraisse class}, it is shown that every tree plan $\Gamma$ gives us a \Fraisse class $\KK(\Gamma)$. Also, it is demonstrated that the common theory of $\KK(\Gamma)$ is  pseudofinite. In Section \ref{K(Gamma) is an Asymptotic Class}, we show that $\KK(\Gamma)$  forms a $\deg(\Gamma)$-dimensional asymptotic class.  Section \ref{Characterization of  aleph0-categorical Trees} provides a characterization theorem for $\aleph_0$-categorical theories of trees.  Section \ref{Characterization of Asymptotic Classes of Trees} offers a characterization of supersimple trees of finite height  and also asymptotic classes of finite trees utilizing the notion of tree plan. Finally, in Section \ref{model theoretic properties of Gamma(omega)}, we study some model theoretic property of the generic model $\Gamma(\omega)$.
 
 \subsection{Notation} Throughout this paper, we use calligraphic upper-case letters like $\M$, $\N$ to denote infinite structures with universes $M$ and $N$, respectively. We use simple upper-case letters like $A,B,C$ to denote finite structures and identify them with their universes. We use bold capital letters like $\KK$, $\mathbf{C}$ to denote class of structures. In general, our notation is standard (see \cite{marker-textbook}). Also, for a formula $\phi$ we write $\qr(\phi)$ to denote the quantifier rank of $\phi$. 
 

 \section{Preliminaries}\label{prelim}
 
 \subsection{Trees} There are many different interpretations of the word ''tree''. Here we will define our notion of tree.
  \begin{defn}\label{defn-tree}
 	\begin{enumerate}
 		\item The \textit{language of trees}, $\LL_t$,  has signature $\sig(\LL_t)= \left\{\leq, \varepsilon, \sqcap, \pred  \right\}$, where $\leq $ is a binary relation symbol, $\varepsilon$ is a constant symbol, $\sqcap$ is a binary function symbol, and $\pred$ is a unary function symbol.

 		\item 	A \textit{tree} is an $\LL_t$-structure $\A$ such that: 
 		\begin{itemize}
 			\item $\leq^\A$ is a partial ordering of the universe in which, for each $b\in A$, the downset $\left\{a:a\leq^\A b\right\}$ is well-ordered.
 			\item $\epsilon^\A$ is the unique minimum element of $\leq^\A$ --- the root of  $\A$.
 			\item For $a,a'\in A$, $a\sqcap^\A a'$ is the unique maximum element of $A$ that is below both $a$ and $a'$ --- the meet of $a$ and $a'$.
 			\item For each $a\in A$, $\pred^\A(a)$ is the maximum element among the elements that are strictly less than $a$. It is called the predecessor of $a$. Also, as a convention we let  $\pred^\A(\varepsilon)=\varepsilon$.
 		\end{itemize}
 		Note that the class of trees in this sense is {\em not} first-order axiomatizable.
 	\end{enumerate}
 \end{defn}
 It follows from Definition \ref{defn-tree} that $\{a : a\leq b\}$ is actually finite for all $b$. The finiteness is because predecessors exist; $b, \pred(b), \pred^2(b), \dots$ is a strictly decreasing sequence, so it must hit $\epsilon$ after finitely many steps, by the well-ordering.
 
 
 Note that $\omega^{<\omega}$ with end-extension ordering has a natural tree structure, and when we write a tree $\Gamma \subseteq \omega^{<\omega}$, we mean that $\Gamma$ is a substructure. In particular, it contains $\gen{}{}$  and is closed under $\pred$. 
 
\subsection{ \Fraisse Limits}
\begin{defn}
	Let $\KK$ be a class of finite (or finitely generated) $\LL$-structures.  We catalog various properties that $\KK$ might have.
	\begin{enumerate}
		\item (\textbf{P0}): for all $n\in \NN$, the set of members of $\KK$ that have generating sets of size $\leq n$ is finite up to isomorphism. 
		\item \textit{Hereditary Property} (\textbf{HP}): 
		for every $B\in\KK$, every induced substructure $A\leq B$ is in $\KK$.
		\item \textit{Joint-Embedding Property} (\textbf{JEP}):
		for any $A,B\in\KK$, there are $C\in \KK$ and embeddings $f_A:A\to C$ and $f_B:B\to C$.
		\item \textit{Amalgamation Property} (\textbf{AP}): 
		for any $A, B_1, B_2\in \KK$ and embeddings $f_i:A\to B_i$, $i=1,2$, there are $C\in \KK$ and embeddings $f'_i:B_i\to C$ such that $f'_1\circ f_1=f'_2\circ f_2$.
		\item \textit{Disjoint Joint-Embedding Property} (\textbf{disjoint JEP}):  for any $A, B\in \KK$ there are $C\in \KK$ and embeddings $f_A:A\to C$ and $f_B:B\to C$ such that $f_A[A]\cap f_B[B]=\emptyset$.
		\item \textit{Disjoint Amalgamation Property} (\textbf{disjoint AP}):  for any $A,B_1,B_2\in \KK$ and embeddings $f_i:A\to B_i $, $i=1,2$, there are $C\in \KK$ and embeddings $f'_i:B\to C$ such that $f'_1\circ f_1= f'_2\circ f_2$ and $f'_1[B_1]\cap f'_2[B_2]=f'_1\circ f_1[A] = f'_2\circ f_2[A]$.
	\end{enumerate}
\end{defn}

\begin{defn}
Let $\KK$ be a class of finite (or finitely generated) $\LL$-structures. $\KK$ is called a \textit{\fraisse class} whenever it is countable up to isomorphism and satisfies \textbf{P0}, \textbf{HP}, \textbf{JEP} and \textbf{AP}. 
\end{defn}

\begin{defn}
	An $\LL$-structure $\M$ is called \emph{homogeneous} if whenever $A, B \in \age(\M)$ with $A\leq B$ and $f : A \to \M$ is an embedding, then there is an embedding $g : B \to \M$   which extends $f$.  (Also, if $\KK$ is a class of finite $\LL$-structures, we say that $\M$ is \emph{$\KK$-homogeneous} if for any $A,B \in \KK$ with $A\leq B$ and any embedding $f_0: A \to \M$, there is an
	embedding $f:B \to M$ such that $f_0\subseteq f$.)
\end{defn}

\begin{defn}
	An $\LL$-structure $\M$ is called \emph{ultra-homogeneous} if for every finite substructure $A\leq \M$ and every embedding $f:A\to \M$, there is an automorphism  $g\in Aut(\M)$ which extends $f$.  (Also, if $\KK\subseteq \age(\M)$ is a class of finite $\LL$-structures, we say that $\M$ is \emph{$\KK$-ultra-homogeneous} if for any $A \in \KK$ and any embedding $f_0: A \to \M$, there is an automorphism  $f\in Aut(\M)$ such that $f_0\subseteq f$.)
\end{defn}
  
  Note that a countable $\LL$-structure $\M$ is ultra-homogeneous ($\KK$-ultra-homogeneous) if and only if it is  homogeneous ($\KK$-homogeneous).

\begin{thm}[\fraisse Theorem]\label{fras}
	Let $\KK\subseteq \age(\M)$ be a \fraisse class of finite $\LL$-structures,  then there is a  generic  model $\M$ with the following properties.
	\begin{itemize}
		\item \textbf{($\KK$-universality)} For every $A\in \KK$, there is an embedding $A\to \M$.
		\item \textbf{($\KK$-ultra-homogeneity)} For any $A\in \KK$  and any embedding $f: A\to \M$, there is an automorphism  $g\in Aut(\M)$ such that $f\subseteq g$.
		\item \textbf{($\KK$-closed)} For every finite $X\subset_{\fin} M$, there is an $A\in \KK$ such that $X\subseteq A\leq \M$.
		\item  ($\KK$-\textbf{uniqueness}) Let $\M, \M'$ be countably infinite $\KK$-universal, $\KK$-homogeneous, and  $\KK$-closed structures. Then $\M\cong \M'$.

	\end{itemize}
\end{thm}
  \begin{proof}
  See Section 11.1 of \cite{fraisse-ThOfRelations} or Chapter 6 of \cite{ShorterModelTheory}.
  \end{proof}
  
The generic model of a \fraisse class $\KK$ is  called the \textit{\fraisse limit }of $\KK$.

Recall that for an infinite cardinal $\kappa$, a first-order theory $T$  is called \textit{$\kappa$-categorical } if it has exactly one, up to isomorphism, model of size $\kappa$. A structure $\M$ is called $\kappa$-categorical if its theory $Th(\M)$ is $\kappa$-categorical. 

\begin{thm}[Ryll-Nardzewski, Engeler and Svenonius]\label{Ryll-Nard-Thm}
	If $T$ is a  countable first-order complete theory in a countable language $\LL$, then the following are
equivalent.
\begin{enumerate}
	\item $T$ is $\aleph_0$-categorical. 
	\item For each $n<\omega$ there are only finitely many $n$-types, i.e. $|S_n(T)|<\infty$ for all $n<\omega$.
	\item For every $n<\omega$,
	there are only finitely many formulas $\phi(x_1, \dots, x_n)$ up to equivalence relative to $T$. 
	\item For every countable model $\M$ of $T$, $\aut(\M)$ is oligomorphic. 
\end{enumerate}
\end{thm}
\begin{proof}
See Theorem 2.3.13 of \cite{ChangKeislerBook}.
\end{proof}

\begin{thm}\label{ultrahom+aleph0-cat---->>QE}.
If $M$ is ultrahomogeneous and $\aleph_0$-categorical, then it eliminates quantifiers.  
\end{thm}
\begin{proof}
	It is enough to show that for every tuples $\aa$ and $\bb$ in any model of $T=Th(\M)$ if $\qftp(\aa)=\qftp(\bb)$, then $\tp(\aa)=\tp(\bb)$. Suppose, towards a contradiction, that there are $\aa, \bb$ in some model $\N\models T$ such that $\tp(\aa)\neq \tp(\bb)$, but $\qftp(\aa)=\qftp(\bb)$. 
	Since $T$ is $\aleph_0$-categorical, there are finitely many formulas in $n$ free variables up to equivalence relative to $T$. So every type is isolated. Thus $\tp(\aa)$ and $\tp(\bb)$ have realizations in $M$. We call those realizations $\aa'$ and $\bb'$. Since $\qftp(\aa')=\qftp(\bb')$, $\M[\aa']\cong \M[\bb']$. By ultra-homogeneity, we get an automorphism $f\in \aut(\M)$ such that $f(\aa')=\bb'$. Therefore $\tp(\aa)=\tp(\bb)$ which is a contradiction.
	\end{proof}

\begin{cor}
The  \fraisse limit of $\KK$ eliminates quatifiers.
\end{cor} 


 \subsection{Asymptotic Classes}
  We recall the definition of asymptotic classes.
 
 \begin{defn}\label{1-dim}
 	Let $\LL$ be a first-order language and let $\KK$ be a class of finite  $\LL$-structures. $\KK$ is called a \emph{1-dimensional asymptotic class} if for  each $\LL$-formula $\phi(x,\yy)$ where $|\yy|=n$, there is  a finite set 
 	$$ \D_\phi=\bigg\{ \big(d_0, \mu_0,\theta_0(\yy)\big), \dots, \big( d_{{k_n}-1},\mu_{{k_n}-1},\theta_{{k_n}-1}(\yy) \big)  \bigg\} \in \P_\fin\bigg(\left\{0,1 \right\}\times [0,\infty)\times  \LL \bigg)$$
 	such that:
 	\begin{enumerate}
 		\item  $\KK\models \bigvee_{i=0}^{k_n-1} \theta_i(\yy)$ and 
 		$\KK\models \big( \theta_i(\yy)\rightarrow \neg\theta_j(\yy)\big)$ for $i\neq j$.
 		
 		\item For $A\in \KK$ and $\bb\in A^n$, if $A\models \theta_i(\bb)$ then $\big| |\phi(A,\bb)|-\mu_i|A|^{d_i}  \big|= \mathrm{o}(|A|^{d_i})$.
 	\end{enumerate}
 	For any number $0 < N < \omega$, we can define the notion of a \emph{$N$-dimensional asymptotic class} by simply  replacing $\left\{0,1 \right\}$  by $\left\{0, \frac{1}{N},\dots, \frac{N-1}{N}, 1\right\}$.
 \end{defn}

\begin{defn}
	Let $T$ be a simple theory. We define the SU-rank or Lascar rank $SU(p)\geq \alpha$ for a type $p$ by induction on $\alpha$:
	\begin{itemize}
		\item $SU(p)\geq 0$ for all types $p$;
		\item $SU(p)\geq \beta+1$ if $p$ has a forking extension $q$ with $SU(q)\geq \beta$;
		\item $SU(p)\geq \lambda$ (for a limit ordinal $\lambda$) if $SU(p)\geq \beta$ for all $\beta<\lambda$. 
	\end{itemize} 
	and the SU-rank of $p$, $SU(p)$, is defined as the maximal $\alpha$ such that $SU(p)\geq \alpha$. If there is no maximum, we set $SU(p)=\infty$.
\end{defn}

\begin{defn}
	A simple theory $T$ is \textit{supersimple} if SU-rank is finite. Equivalently, $T$ is supersimple whenever for every type $p\in S(A)$, there is a  \emph{finite}   subset $A_0\subseteq A$ such that $p$ does not fork over $A_0$. 
\end{defn}

See \cite{cassanovas}, \cite{Wagner-SimpleTheories}, and \cite{KimBook} for more details about super-simplicity and SU-rank.

\begin{prop}\label{supersimple finite su-rank}
Any infinite ultraproduct of an N-dimensional asymptotic class is  supersimple of SU-rank at most N.
\end{prop}  
\begin{proof}
See Lemma 4.1 in \cite{macpherson-steinhorn-2007}.
\end{proof}

 
 \section{Tree Plans}\label{sec-treeplan}
 
 In order to study the theory of $\aleph_0$-categorical trees, we first need the notion of tree plan.

 \begin{defn}\label{def_tree-plan_I(S)} 
 	\begin{enumerate}
 		\item 
 		
 		A \textit{tree plan} is just a pair $(\Gamma,\lambda)$,   where $\Gamma\subset_\fin\omega^{<\omega}$ is a finite tree (containing $\gen{}{}$) and $\lambda:\Gamma\to\{1,\infty\}$ is a function such that $\lambda(\gen{}{}) = 1$.  
 		
 		\noindent	Under almost all circumstances it's preferable to omit $\lambda$ from the notation and start from the declaration ``$\Gamma = (\Gamma,\lambda)$ is tree plan...''

 		\item \label{I(S)}
 		Given a tree plan $\Gamma = (\Gamma,\lambda)$, we define 
 		$$I(\Gamma) := \lambda^{-1}[\infty] = \left\{\sigma\in\Gamma:\lambda(\sigma) = \infty\right\}.$$
 		We call each member of $I(\Gamma)$ an \textit{infinity-node} of $\Gamma$. 
 	\end{enumerate}
 \end{defn}
 \medskip 
 
 For a given tree plan $\Gamma$ and a set $X$ (possibly infinite), we can define a tree, called $\Gamma(X)$,  by identifying every infinity-node of $\Gamma$ with the set $X$.  Roughly speaking, $\Gamma(X)$ is an extension of $\Gamma$ obtained by replacing each infinity-node with $|X|$-many new nodes, and leaving other nodes to themselves.
 
 \begin{defn}  
 	Let $\Gamma = (\Gamma,\lambda)$ be a tree plan. For any non-empty set $X$, we define a tree 
 	$$\Gamma(X)\subseteq\left(\omega\times\big(X\cup\{\filledstar\}\big)\right)^{<\omega}$$ 
 	and a function $\pi_X:\Gamma(X)\to\Gamma$ as follows:
 	\item $\gen{ (i_0,t_0),...,(i_n,t_n)}{}\in\Gamma(X)$ if $\gen{i_0,...,i_n}{}\in\Gamma$ and for each $k\leq n$:
 	\begin{enumerate}
 		\item if $\lambda(\gen{i_0,...,i_k}{})=1$, then $t_k=\filledstar$;
 		\item  if $\lambda(\gen{i_0,...,i_k}{})=\infty$, then $t_k\in X$.
 	\end{enumerate}
 	We then define $\pi_X$ just by setting $\pi_X{\left(\gen{ (i_0,t_0),...,(i_n,t_n)}{}\right)} = \gen{i_0,...,i_n}{}$, and we define 
 	$$I(X) = \left\{a\in\Gamma(X): \pi_X(a)\in I(\Gamma)\right\} =  \left\{a\in\Gamma(X): \lambda(\pi_X(a))=\infty\right\}$$
 	which is the set of all infinity-nodes of $\Gamma(X)$. Note that for each $\sigma\concat i\in I(\Gamma)$ and $b\in \Gamma(X)$ such that $\pi_X(b) = \sigma$, there is a bijection between $X$ and the set 
 	$$\left\{a\in \Gamma(X): \pred(a) = b \,\,\,\wedge \,\,\,\pi_X(a) = \sigma\concat i\right\}.$$
 \end{defn}
 
 \noindent	For the sake of brevity, for $n<\omega$, we write $\Gamma(n)$ in place of $\Gamma([n])$ and $\pi_n$ in place of $\pi_{[n]}$. 
 

 \begin{obs} 
 	Let $\Gamma = (\Gamma,\lambda)$ be a tree plan. Then for any two sets $X$ and $Y$, if $X\subseteq Y$, then $\Gamma(X)\leq \Gamma(Y)$ and $\pi_X = \pi_Y\r \Gamma(X)$.
 \end{obs}
 
 \begin{defn}\label{defn_downward-closed} 
 	For a tree plan $\Gamma$, a node $\sigma\in \Gamma$, a set $X$ and $B\subseteq \Gamma(X)$,
 	\begin{enumerate}
 		\item  We define ${\downarrow} B=\left\{a\in \Gamma(X) : (\exists b\in B) \, a\leq b\right\}$. We say $B$ is \textit{downward closed} if $B={\downarrow}B$.
 		\item We define the \emph{height }of $\sigma$, $\h(\sigma)$, to be the smallest number $k<\omega$ such that $\pred^k(\sigma)=\epsilon$,  i.e.
 		\begin{itemize} 
 			\item $\h(\gen{}{}) = 0$,
 			\item if $\sigma\neq \gen{}{}$, then $\h(\sigma)=k$ if and only if $\pred^k(\sigma)=\epsilon$ and $\pred^{k-1}(\sigma)\neq \epsilon$.
 		\end{itemize}
 		Also, we define the \emph{height of a tree plan} (or in general the height of an arbitrary tree) $\Gamma$ as follows: 
 		$$\h(\Gamma)=\max \left\{ \h(\sigma) : \sigma\in \Gamma\right\}. $$
 	\end{enumerate}
 \end{defn} 
 
 
 
 \begin{defn}\label{defn_tree-closure} 
 	Let $\Gamma = (\Gamma,\lambda)$ be a tree plan. For a set $X$, we define a operator, called \textit{tree closure},
 	$$\tcl=\tcl^{\Gamma(X)}:\P(\Gamma(X))\longrightarrow\P(\Gamma(X))$$ 
 	as follows:
 	\begin{align*}
 	\tcl_0(B) & = \Big\{ \gen{}{} \Big\} \cup \Big\{ a\in \Gamma(X) \, : \, (\exists x\in B)\,\, a\leq x \Big\} = \Big\{ \gen{}{} \Big\} \cup {\downarrow}B,\\
 	\tcl_{n+1}(B) & = \tcl_n(B) \cup \Big\{ a\in \Gamma(X) \, : \, \lambda(\pi(a))=1 \text{ and }  \pred(a)\in \tcl_n(B) \Big\},\\
 	&\\
 	\tcl(B) & = \bigcup_n \tcl_n(B).
 	\end{align*}
 	In other words, for a given $B\subseteq \Gamma(X)$,  $\tcl^{\Gamma(X)}(B)$ is the smallest subset of $\Gamma(X)$ containing ${\downarrow}B\cup \{\gen{}{} \}$ such that for all $a, a'\in \Gamma(X)$, if 
 	$$\textnormal{$a\in \tcl^{\Gamma(X)}(B)$, $\pred(a') = a$, and $\lambda(\pi_X(a')) = 1$,}$$
 	then $a'\in \tcl^{\Gamma(X)}(B)$, too. The set $B$ is said to be \textit{tree-closed} whenever $B=\tcl(B)$.
 	
 \end{defn}
 
 Using the tree closure operator, we formulate another definition which is useful in relating an element $a$ to $\tcl^{\Gamma(X)}(B)$ that may not itself be a member of $\tcl^{\Gamma(X)}(B)$. For $a\in\Gamma(X)$ and $B\subseteq \Gamma(X)$, we set 
 $$[a{\wedge} B; X]^\Gamma = \max\left\{e\in\tcl^{\Gamma(X)}(B):e\leq a\right\}.$$
 Note that $a\in\tcl^{\Gamma(X)}(B)$ if and only  $[a{\wedge} B; X]^\Gamma = a$.
 
 \noindent \begin{rem}
 	From now on, we fix a tree plan $\Gamma= (\Gamma,\lambda)$.  Let $$\sig(\LL_\Gamma) =  \sig(\LL_t)\cup\left\{P_\sigma^{(1)}:\sigma\in\Gamma\right\},$$ 
 	where for each $\sigma\in \Gamma$, by $P_\sigma^{(1)}$ we mean that $P_\sigma$ is a unary relation symbol. 
 	We consider $\Gamma(\omega)$ as an $\LL_\Gamma$-structure with $P_\sigma^{\Gamma(\omega)} = \left\{a:\pi(a) = \sigma\right\}$ for each $\sigma\in\Gamma$.  
 \end{rem} 
  
 \section{$\KK(\Gamma)$ is a \fraisse Class}\label{K(Gamma) is a Fraisse class}
 Suppose $\Gamma$ is a tree plan. We define the following notions.
 \begin{defn}
 	We define the number $\ell(\Gamma)$ to be  
 	one more than the maximum number of non-$\infty$-successors of a node in $\Gamma$ -- i.e. 
 	$$\ell(\Gamma) = 1+\max\left\{\,\Big|\left\{i\in\omega:\lambda(\sigma\concat i)=1\right\}\Big|:\sigma\in\Gamma\right\}.$$
 \end{defn} 
 
 \begin{defn}\label{Def-K(Gamma)} 
 	We define a class of finite trees, $\KK(\Gamma)$, as follows 
 	$$\KK(\Gamma) = \left\{A\in\TT_\Gamma:(\exists n\in \NN^+)\,A\cong\Gamma(n)\right\}.$$
 	where $\TT_\Gamma$ is the class of all finite $\LL_\Gamma$-structures $A$ such that $A\r\LL_t$ is a tree.
 \end{defn}
 
 \begin{obs}
 	$\KK(\Gamma)$ has JEP.  
 \end{obs}
 \begin{proof}
 	For every $n, m\in \NN^+$, $\Gamma(n)$ and $\Gamma(m)$ embed in $\Gamma(n+m)$ in a natural way.
 \end{proof}	
 
 \subsection{$\KK(\Gamma)$ Has the Amalgamation Property (AP)}
 
 \begin{lemma}\label{extend_pi_embeddings}
 	Let $X$ be an arbitrary non-empty set.
 	Let $B \subseteq \Gamma(X)$ and $c\in  \Gamma(X)$ such that $\pred(c) \in \tcl(B)$, and let $f_0 : \tcl(B) \to \Gamma(\omega)$ be an embedding. Then there is an embedding $f: \tcl(Bc) \to \Gamma(\omega)$ such that $f_0\subseteq f$. 
 \end{lemma}  
 \begin{proof} 
 	Let $f_0 : \tcl(B)\to \Gamma(X)$ be an embedding.
 	\begin{itemize}
 		\item If $\lambda(\pi(c))=1$, then $c\in \tcl(B)$ and so $\tcl(B)=\tcl(Bc)$. In this case it is enought to put $f=f_0$.
 		\item Suppose $\lambda(\pi(c))=\infty$.  We need to guarantee that the extended embedding $f$ satisfies the following condition:
 		$$ f_0(\pred(c))=\pred(f(c)).$$
 		Let $f_0(\pred(c))=d$. There are $|X|$-many points in $\Gamma(X)$, say $e_1, \dots, e_{|X|}$, such that for each $1\leq i\leq |X|$ we have
 		$\pred(e_i)=d$ and
 		$\pi(e_i)=\pi(c)$.	Let $I=\big\{ i : e_i\in \img(f_0) \big\}$.
 		
 		\begin{claim}
 			$I\ne \big\{ 1, \dots, |X| \big\}$
 		\end{claim}
 		\begin{proof}
 			If 			$I = \big\{ 1, \dots, |X| \big\}$, then $c\in f_0^{-1}\big[\{e_1, \dots, e_{|X|} \}\big]$. This implies that $c\in \tcl(B)$ which is a contradiction. 
 		\end{proof}
 		Now, let $k=\min\big\{ 1, \dots, |X| \big\}\setminus I $. 	We define $f:\tcl(Bc)\to \Gamma(X)$ such that $f(c)=e_k$. When we fix the image of $c$ under $f$ then we look at the tree above $c$ and map each point to its corresponding node above $f(c)$. 
 	\end{itemize}
 \end{proof}
 
 \begin{lemma}\label{rearrange}
 	Let $m\leq n<\omega$, and suppose $h:\Gamma(m)\to\Gamma(n)$ is an embedding of $\LL_\Gamma$-structures.
 	\begin{enumerate}
 		\item For all $a\in\Gamma(m)$, $\pi_m(a)=\pi_n(h(a))$.
 		\item There is an automorphism $g\in Aut(\Gamma(n))$ such that $g\circ h = id_{\Gamma(m)}$.
 	\end{enumerate}
 \end{lemma} 
 \begin{proof}
 	Let an embedding $f: \Gamma(m)\to \Gamma(n)$ be given.
 	First, note that $\Gamma(m)\leq \Gamma(n)$ and $\Gamma(m)$ is $\tcl$-closed, i.e. $\tcl\big(\Gamma(m) \big)=\Gamma(m)$. 
 	Pick a point $a\in \Gamma(n)\setminus \Gamma(m)$ such that $\pred(a)\in \Gamma(m)$. Apply Lemma \ref{extend_pi_embeddings} and get an embedding $f_1: \tcl\big(\Gamma(m)a \big)\to\Gamma(n)$. Now, pick an element $b\in \Gamma(n)\setminus \tcl\big(\Gamma(m)a \big) $ but $\pred(b)\in \tcl\big(\Gamma(m)a \big)$. Again we apply Lemma \ref{extend_pi_embeddings} and get a an embedding $f_2:\tcl\big(\Gamma(m)ab \big)\to \Gamma(n)$. 	We iterate this process until we run out of elements to put in the domain. So after finitely many steps we get an embedding $g: \Gamma(n)\to \Gamma(n)$. Since $\Gamma(n)$ is finite, $g$ must be surjective as well. Hence  $g\in Aut\big( \Gamma(n)\big)$ and $f\subseteq g$.
 \end{proof}
 
 \begin{prop}\label{disjoint-AP}
 	$\KK(\Gamma)$ has disjoint-AP.
 \end{prop}
 \begin{proof}
 	Let $A_0, A_1, A_2\in \TT(\Gamma)$ and embeddings $f_i: A_0\to A_i$ for $i=1, 2$ be given. There are $n_i\in \NN^+$ and isomorphisms $u_i: A_i\to \Gamma(n_i)$ for $i<3$. By Lemma \ref{rearrange}   there are automorphisms $g_i\in Aut(\Gamma(n_i))$  such that $g_i\circ\big(u_i\circ f_i\circ u_0^{-1} \big) =id_{\Gamma(n_0)}$ for $i=1,2$. Now let $B=\Gamma(n)$ where $n\geq n_1+n_2$.  Let $\iota_1: n_1\to n$ be the inclusion map, and $\iota_2:n_2\to n$ be the following map: 
 	$$k\mapsto \begin{cases} 
 	k & k\leq n_0 \\
 	k+n_1 & \text{otherwise} 
 	\end{cases}$$ This gives us embeddings $\hat{\iota}_i: \Gamma(n_i)\to \Gamma(n)$ for $i= 1, 2$ such that $\hat{\iota}_1[\Gamma(n_1)]\cap \hat{\iota}_2[\Gamma(n_2)]=\Gamma(n_0)$.   So if we choose $B=\Gamma(n)$ as the amalgam, and consider the embeddings $j_i:=\hat{\iota}_i \circ g_i\circ u_i: A_i\to B$,   then $j_1[A_1]\cap j_2[A_2]=j_1\circ f_1 [A_0]=j_2\circ f_2[A_0]$, for $i=1,2$.  
 	
 	\begin{center}
 		\begin{tikzcd} & A_1 \arrow[r, "\cong", "u_1"' ] & \Gamma(n_1) \arrow[r, "\cong", "g_1"' ] &  \Gamma(n_1) \arrow[dr, "\hat{\iota}_1" ] \\
 		A_0  \arrow[r, "\cong", "u_0"']  \arrow[ur, "f_1"] \arrow[dr,"" ,"f_2"'] & \Gamma(n_0)   \arrow[ur, "u_1\circ f_1\circ u_0^{-1}"']  \arrow[dr,"" ,"u_2\circ f_2\circ u_0^{-1}"] & & & B:=\Gamma(n)  \\ 
 		& A_2 \arrow[r, "\cong", " u_2 "'] &\Gamma(n_2) \arrow[r, "\cong", "g_2"'] & \Gamma(n_2) \arrow[ur, "\hat{\iota}_2"']  
 		\end{tikzcd} 
 	\end{center}
 \end{proof}
 
 
 \begin{obs}\label{permutation} 
 	Let $\Gamma = (\Gamma,\lambda)$ be a  tree plan. For any set $X$ and any permutation $f\in\sym(X)$, $f$ induces an automorphism $\hat{f}\in \aut(\Gamma(X))$. 
 \end{obs}

 \begin{notation}
 	For brevity, we define
 	$Th(\Gamma):= Th(\Gamma(\omega))$.
 \end{notation}
 
 \begin{prop}\label{Gamma(S)_generic}
 	$\Gamma(\omega)$ is a copy of the generic model of $\KK(\Gamma)$. 
 \end{prop}
 \begin{proof}
 	We show that $\Gamma(\omega)$ has the following three properties:
 	\begin{itemize}
 		\item $\KK(\Gamma)$-universality: for every $A \in \TT(\Gamma)$, there is an embedding $A \to \Gamma(\omega)$.
 		\item $\KK(\Gamma)$-homogeneity: for any $A, B \in  \TT(\Gamma)$  with $A\leq B$ and any embedding $f: A \to \Gamma(\omega)$, there is an
 		embedding $g: B\to \Gamma(\omega)$ such that $f\subseteq g$.
 		\item $\KK(\Gamma)$-closedness: for every finite $X \subset_\fin \Gamma(\omega)$,  there is an $A \in \TT(\Gamma)$ such that $X \subseteq A\leq \Gamma(\omega)$.
 	\end{itemize}
 	
 	For $\KK(\Gamma)$-universality, let $A\in \TT(\Gamma)$. Then for some $n\in \NN^+$ there is an isomorphism $f: A\cong \Gamma(n) $.  Let $g:n\to S$ be an injection map. Then $f$ induces an embedding $\hat{g}: \Gamma(n)\to \Gamma(\omega)$. Hence  $\hat{g}\circ f: A\to \Gamma(\omega)$ is an embedding. 
 	
 	For $\KK(\Gamma)$-homogeneity, let $A, B\in \TT(\Gamma)$, $A\leq B=\Gamma(n)$ for some $n\in \NN^+$, and let $f: A\to \Gamma(\omega)$ be a embedding. First, we observe that in $\Gamma(n)$, $\tcl(A)=A$.  Pick a point $a_1\in B\setminus A$ such that $\pred(a_1)\in A$. By Lemma \ref{extend_pi_embeddings}, there is a embedding $f_1:\tcl(Aa_1)\to \Gamma(\omega)$ such that $f\subset f_1$. Now, take a point $a_2\in B\setminus \tcl(Aa_1)$ such that $\pred(a_2)\in \tcl(Aa_1)$. Again we apply Lemma \ref{extend_pi_embeddings} to get a embedding $f_2: \tcl(Aa_1a_2)\to \Gamma(\omega)$ such that $f_1\subset f_2$. We do this process untill we run out of elements to put in the domain. Since $B$ is finite, after finitely many iterations we get a embedding $g:B \to \Gamma(\omega)$ which extends $f$. 
 	
 	For $\KK(\Gamma)$-closedness, let $X\subset_\fin \Gamma(\omega)$. There is a finite subset $S_X\subset_\fin S$ which is big enough such that $X\subseteq\Gamma(S_X)$. Hence if we put $A=\Gamma(S_X)$, then $X \subseteq A\leq \Gamma(\omega)$.
 \end{proof}

 \begin{cor} 
 	$Th(\Gamma)$ is $\aleph_0$-categorical and has quantifier elimination.
 \end{cor}
 \begin{proof}
 	By {\Fraisse} Theorem \ref{fras} and Theorem \ref{ultrahom+aleph0-cat---->>QE}.
 \end{proof}
 
 \subsection{Ehrenfeucht–{\Fraisse} games on $\Gamma(X)$'s}\label{section-EF-game} 
 The Ehrenfeucht–{\Fraisse} game (for brevity EF game) is a method for determining whether two structures are elementarily equivalent 
 (for more details see \cite{spencer-2001} and \cite{book-flum}). EF game has many applications in finite model theory (see \cite{book-flum}). 
 \begin{fact}[see \cite{big-hodges}]
 	Let $\M$ and $\N$ be two structures of the same language. Then the following statements are equivalent:
 	\begin{enumerate}
 		\item $\M$ and $\N$ satisfy the same first-order \textit{unnested} sentences of quantifier rank $n$: i.e., $\M\equiv_n \N$.
 		\item Player II has a winning strategy for $n$ rounds of the EF game played on $\M$ and $\N$.
 	\end{enumerate}
 	See Section 2.6 of \cite{big-hodges} for definition of unnested sentence. Also for more details see Corollary 2.6.2 and Theorem 3.3.2 of \cite{big-hodges}.  
 \end{fact}
 
 \begin{prop}\label{N-Gamma}
 	There is a function $N_\Gamma:\omega\to\omega$ such that for every $k$, for any finite set $X$, if $|X|\geq N_\Gamma(k)$, then $\Gamma(X)\equiv_k\Gamma(\omega)$ --- i.e. Player II has a winning strategy for $k$ rounds of the EF game played on $\Gamma(X)$ and $\Gamma(\omega)$. 
 \end{prop}
 \begin{proof}  
 	For a given tree plan $\Gamma$, let's define $N_\Gamma:\omega\to \omega $ with $k\mapsto k+\ell(\Gamma)\cdot \h(\Gamma)$. We show that for every $k\in \NN^+$, if $|X|\geq N_\Gamma(k)=k+\ell(\Gamma)\cdot \h(\Gamma)$, then $\Gamma(X)\equiv_k \Gamma(\omega)$.   Let $f: X\to S$ be an injection. This  gives us an embedding $\hat{f}: \Gamma(X)\to \Gamma(\omega)$. Player I can  pick either a point in $\Gamma(X)$  or $\Gamma(\omega)$. 
 	\begin{itemize} 
 		\item In round 1, we have the following cases.
 		
 		\begin{itemize}
 			\item If Player I chooses  $a_1\in \Gamma(X)$, then Player II can choose $b_1:= \hat{f}(a_1)\in \Gamma(\omega)$. 
 			\item Suppose Player I chooses $b_1\in \Gamma(\omega)$. 
 			\begin{itemize}
 				\item If $b_1\in \img(\hat{f})$, then Player II can choose $a_1:=\hat{f}^{-1}(b_1)$.
 				\item Suppose $b_1\in \Gamma(\omega)\setminus \img(\hat{f})$. Let  $X'\subset_\fin S$  be a subset such that $|X'|=|X|$ and $b_1\in \Gamma(X')$. Let $\beta:f[X]\to X'$ be a bijection and $\hat{\beta}:\Gamma(f[X])\to \Gamma(X')$ be the  isomorphism obtained from $\beta$. Let $g\in Aut(\Gamma(\omega))$ be an automorphism extending $\hat{\beta}$. Then Player II can choose  $a_1:=\hat{f}^{-1}(g^{-1}(b_1))$.
 			\end{itemize} 
 		\end{itemize} 
 		So we end up with $a_1\in \Gamma(X), b_1\in \Gamma(\omega)$ and an embedding $f_1:\tcl(a_1)\to \Gamma(\omega)$ such that $f_1(a_1)=b_1$.
 		
 		\item 	At the end of round $k-1$ we end up with $a_1, \dots a_{k-1}\in \Gamma(X)$,  $b_1,\dots, b_{k-1}\in \Gamma(\omega)$ and an embedding $f_{k-1}: \tcl(a_1,\dots, a_{k-1})\to \Gamma(\omega)$ such that $f_{k-1}(a_i)=b_i$ for each $i=1,\dots, k-1$.
 		 
 		\item  In round $k$, we consider the following cases.
 		\begin{itemize}
 			\item If  Player I chooses $a_k\in \Gamma(X)\setminus \{a_1,\dots, a_{k-1}\}$, then we consider the following two cases.
 			\begin{itemize}
 				\item If $a_k\in \tcl\big(a_1,\dots, a_{k-1}\big)$, then Player II can choose $b_k:=f_{k-1}(a_k)$.
 				\item If $a_k\notin \tcl\big(a_1,\dots, a_{k-1}\big)$,  then  by homogeneity we extend  $f_{k-1}$ to an embedding $f_{k-1}^{\ast}:\Gamma(X)\to \Gamma(\omega)$, so Player II can choose $b_k:=f^{\ast}_{k-1}(a_k)$.
 				
 			\end{itemize}
 			
 			\item If Player I choose $b_k\in \Gamma(\omega)\setminus \{b_1,\dots, b_{k-1}\}$, then we consider the following two cases:
 			\begin{itemize}
 				\item  If $b_k\in \img(f_{k-1})$, then Player II can choose $a_k:=f_{k-1}^{-1}(b_k)$.
 				
 				\item If $b_k\notin \img(f_{k-1})$, then let $t\in \NN^+$ be the smallest number such that $d:=\pred^t(b_k)\in  \img(f_{k-1})$,  i.e. $d =[b_k\wedge\img(f_{k-1})]$ is the most recent ancestor of $b_k$ which is in $\img(f_{k-1})$.  Now let $p(x,y)=\qtp(b_k, d)$ and $c:=f_{k-1}^{-1}(d)$. Since $b_k\notin \img(f_{k-1})$ and $\tcl\big(\img(f_{k-1})\big)=\img(f_{k-1})$, $\big| p\big(\Gamma(\omega), d\big)\big|=\infty$. Since $|X|\geq k+\ell(\Gamma)\cdot\h(\Gamma)$, $p\big(\Gamma(X), c\big)\setminus \tcl(a_1, \dots, a_{k-1})\neq \emptyset $. So Player II can choose $a_k$ to be a member of  $p\big(\Gamma(X), c\big)\setminus \tcl(a_1, \dots, a_{k-1})$.
 			\end{itemize}
 		\end{itemize}
 	\end{itemize}
 \end{proof}

 \begin{prop}
 	Let $\Gamma$ be a tree plan. 
 	For some $n<\omega$, let $p_0(\xx),...,p_{m-1}(\xx)$ be an enumeration of the complete $n$-types of $Th(\Gamma)$, and for each $i<m$, let $\theta_i(\yy)$ be a principal formula of $p_i$.  There is a number $K_n$ such that for any finite sets $X\subset_\fin S$ and any $\bb\in \Gamma(X)^n$, if $|X|\geq K_n$, $\Gamma(X)\models\theta_i(\bb)\iff \Gamma(\omega)\models \theta_i(\bb)$ for each $i<m$. 
 \end{prop}
 
 \begin{proof}
 	Let $\qr(\theta_i)=r_i$ for  $i<m$.   Let $r^\ast = \max\left\{r_i\right\}_{i<m}$. By  Proposition \ref{N-Gamma}  there is a function $N_\Gamma:\omega\to \omega$ such that for every finite set $X\subset_\fin S$, if $|X|\geq N_\Gamma(k)$, then $\Gamma(X)\equiv_{k} \Gamma(\omega)$. Let $K_n=N_\Gamma(r^\ast)$, and let $X\subset_\fin S$ be given such that  $|X|\geq N_\Gamma(r^\ast)=K_n$. Then $\Gamma(X)\equiv_{r^\ast}\Gamma(\omega)$. Since $r^\ast = \max\left\{r_i\right\}_{i<m}$, for any $\bb\in \Gamma(X)^n$, $\Gamma(X)\models\theta_i(\bb)\iff \Gamma(\omega)\models \theta_i(\bb)$ for each $i<m$.
 \end{proof}

 \subsection{Pseudo-finiteness of $\Th(\Gamma)$}  
 
 \begin{thm}
 	$Th(\Gamma)$ is $\KK(\Gamma)$-pseudofinite.
 \end{thm}
 \begin{proof}
 	This follows  from the EF game result in Section \ref{section-EF-game}. For a given sentence $\phi$, let $\hat{\phi}$ be an unnested sentence which is logically equivalent to $\phi$. Let $\qr(\hat{\phi})=k$ and let $X$ be a finite set with $|X|\geq N_\Gamma(k)$.  By Proposition \ref{N-Gamma}, $\Gamma(X)\models \hat{\phi}$. Hence $\Gamma(X)\models \phi$. 
 \end{proof}
 
 
 \section{$\KK(\Gamma)$ is an Asymptotic Class}\label{K(Gamma) is an Asymptotic Class} 
 
  For all of this section, we fix a tree plan $\Gamma = (\Gamma,\lambda)$. By defining appropriate polynomials associated with $\Gamma$, we prove asymptotic estimates for complete 1-types and then lift asymptotic estimates to arbitrary formulas.

 
 \subsection{Polynomials Associated with Tree Plans}
 \begin{defn}
 	Let $(\Gamma,\lambda)$ be a tree plan, and let $\sigma\in\Gamma$. We define a another tree plan $(\Gamma_\sigma,\lambda_\sigma)$ as follows:
 	\begin{itemize}
 		\item $\Gamma_\sigma = \left\{\tau:\sigma\concat \tau\in\Gamma\right\}$ 
 		\item $\lambda_\sigma(\gen{}{}) = 1$ and for $\tau\in\Gamma_\sigma$ different from $\gen{}{}$, $\lambda_\sigma(\tau) = \lambda(\sigma\concat \tau)$.
 	\end{itemize}
 \end{defn}
 
 
 \begin{defn}\label{I(S,sigma)}
 	Given a tree plan $\Gamma = (\Gamma,\lambda)$ and a node $\sigma\in \Gamma$, we define 
 	$$I(\Gamma, \sigma) = \big| \left\{ \tau\in I(\Gamma) : \tau\leq \sigma \right\}  \big|$$ which is the set of all $\infty$-nodes less than or equal to  $\sigma$.
 \end{defn}
 
 
 \begin{defn}
 	For each tree plan $\Gamma = (\Gamma,\lambda)$, we define a polynomial $P(\Gamma;\,x)\in \ZZ[x]$ with  non-negative coefficients. 
 	\begin{itemize}
 		\item $P(\gen{}{};\,x) = 1$.
 		\item Let $\Gamma = (\Gamma,\lambda)$ be a tree plan, and let  $\sigma_0,\dots  , \sigma_{n-1}$ be the successors of $\gen{}{}$ in $\Gamma$; for each $i<n$, let $f_i(x) = 1$ if $\lambda(\sigma_i) = 1$ and $f_i(x) = x$ if $\lambda(\sigma_i) = \infty$. Then 
 		$P(\Gamma;\,x) = 1+\sum_{i<n}f_i(x)P(\Gamma_{\sigma_i};\,x)$
 	\end{itemize}
 	For each tree plan $\Gamma = (\Gamma,\lambda)$, we also define $\deg(\Gamma) = \deg{\big(P(\Gamma;\,x)\big)}$, and we define $A(\Gamma)$ to be the coefficient of the maximum degree monomial of $P(\Gamma;\,x)$.  So $P(\Gamma;\, x)=A(\Gamma)x^{\deg(\Gamma)}+\dots$.
 \end{defn}
 
 \begin{lemma}
 	Let $\Gamma = (\Gamma,\lambda)$ be a tree plan. Then, for every finite set $X$, $|\Gamma(X)| = P(\Gamma;\,|X|)$.
 \end{lemma}
 
 \begin{proof}
 	By induction on the height of $\Gamma$, $\h(\Gamma)$, we show that $|\Gamma(X)| = P(\Gamma;\,|X|)$. Let $\Gamma$ be a tree plan, and  let  $\sigma_0, \dots, \sigma_{n-1}$ be an enumeration of the successors of $\gen{}{}$ in $\Gamma$. Let $X$ be a finite set. 
 	\begin{itemize}
 		\item If $\h(\Gamma)=0$, then $\Gamma=\gen{}{}$. So $\Gamma(X)\cong \gen{}{}$. Thus $|\Gamma(X)|=1=P(\gen{}{};x)$.
 		\item Suppose $\h(\Gamma)=k+1$ and  for every tree plan $\Gamma'$ with $\h(\Gamma')=k$ we have $|\Gamma'(X)|=P(\Gamma';|X|)$.
 		\item To build $\Gamma(X)$, for each $i<n$ we have one of the following  cases:
 		\begin{itemize}
 			\item if $\lambda(\sigma_i)=1$, then by putting $\Gamma_{\sigma_i}(X)$ as the successor of the  root we add $\vert\Gamma_{\sigma_i}(X)\vert-1$ many new nodes.
 			\item if $\lambda(\sigma_i)=\infty$, then  we actually add $|X|$-many copies of $\Gamma_{\sigma_i}(X)$ to the root.
 		\end{itemize}
 		So $$|\Gamma(X)|=1+\sum_{i<n} f_i(|X|)P(\Gamma_{\sigma_i};|X|)=P(\Gamma;|X|). $$ 
 	\end{itemize}
 \end{proof}
 
 \begin{defn}
 	We define a family of polynomials $Q(\Gamma,\sigma;\,x)\in\ZZ[x]$ with non-negative coefficients, where $\Gamma =(\Gamma,\lambda)$ is a tree plan and $\sigma\in\Gamma$.
 	\begin{itemize}
 		\item $Q(\Gamma,\gen{}{};\,x) = 1$.
 		\item Let $\Gamma = (\Gamma,\lambda)$ be a tree plan, and let $\sigma\concat i\in\Gamma$. 
 		\noindent Then: 
 		\begin{itemize}
 			\item if $\lambda(\sigma\concat i) = 1$, then
 			$Q(\Gamma,\sigma\concat i;\,x ) =  Q(\Gamma,\sigma;\,x)$;
 			\item if $\lambda(\sigma\concat i) = \infty$, then
 			$Q(\Gamma,\sigma\concat i;\,x ) =  x\cdot Q(\Gamma,\sigma;\,x)$.
 		\end{itemize} 
 	\end{itemize}
 	So $Q(\Gamma, \sigma;\, x)= x^{I(\Gamma, \sigma)}$ where $I(\Gamma, \sigma)$ is defined in Definition \ref{I(S,sigma)}.
 \end{defn}

 
 \begin{lemma}\label{Q(..,|X|)}
 	Let $\Gamma = (\Gamma,\lambda)$ be a tree plan, and let $\sigma\in\Gamma$. Then, for every finite set $X$, 
 	$$Q(\Gamma,\sigma;\,|X|) = \Big|\Big\{a\in\Gamma(X):\pi(a) = \sigma\Big\}\Big|.$$
 \end{lemma}
 
 \begin{proof}
 	We do induction on the length of $\sigma$. 
 	\begin{itemize}
 		\item If $\sigma=\gen{}{}$, then $Q(\Gamma, \gen{}{};x)=1$, and $\left\{a\in \Gamma(X): \pi(a)=\gen{}{}\right\}=\left\{\gen{}{}\right\}$. Thus $$Q(\Gamma, \gen{}{};|x|)=1=\left|\left\{ a\in \Gamma(X) : \pi(a)=\gen{}{} \right\}\right|.$$
 		\item Let $\sigma=\tau\concat i$,  and assume that for each $\gamma\in \Gamma$ with $\gamma \subset \sigma$, 
 		$$Q(\Gamma, \gamma;\, |X|)= |\left\{ a\in \Gamma(X) : \pi(a)=\gamma \right\}|.$$ 
 		We have two cases:
 		\begin{itemize}
 			\item If $\lambda(\sigma)=\lambda(\tau\concat i)=1$, then 
 			\begin{align*}
 			Q(\Gamma, \sigma; |X|) & =  Q(\Gamma, \tau; |X|)\\
 			& = |\left\{ a\in \Gamma(X) : \pi(a)= \tau \right\}| \\
 			& =|\left\{ a\in \Gamma(X) : \pi(a)= \sigma \right\}|.
 			\end{align*}
 			
 			\item If $\lambda(\sigma)=\infty$, then  
 			\begin{align*}
 			Q(\Gamma, \sigma;\,|X|) & = |X|\cdot Q(\Gamma, \tau;\, |X|) \\ & = |X|\cdot |\left\{ a\in \Gamma(X) : \pi(a)
 			=\tau \right\}| \\
 			& =|\left\{ a\in \Gamma(X) : \pi(a)
 			=\sigma \right\}|.
 			\end{align*}
 			
 			The last equality holds because for each member of $\left\{ a\in \Gamma(X) : \pi(a)=\tau \right\}$ we have $|X|$ many choices $a\in \Gamma(X)$ such that  $\pi(a)=\sigma$.
 		\end{itemize}
 	\end{itemize}
 \end{proof}
 
 \begin{defn} 
 	We define yet another family of polynomials  $Q(\Gamma,\sigma,\sigma';\,x)\in\ZZ[x]$ with non-negative coefficients, where $\Gamma =(\Gamma,\lambda)$ is a tree plan and $\sigma\subseteq\sigma'\in\Gamma$. Given such $\Gamma,\sigma,\sigma'$, if $\sigma' = \sigma\concat \tau$, then 
 	$$Q(\Gamma,\sigma,\sigma';\,x) := Q(\Gamma_\sigma,\tau;\,x).$$
 	In terms of $Q(\Gamma,\sigma,\sigma';\,x)$, we then define   
 	\begin{itemize}
 		
 		\item $\deg_\Gamma(\sigma'/\sigma) = \deg{\big(\,Q(\Gamma,\sigma,\sigma';\,x)\,\big)}.$  
 		
 		\item $\deg_\Gamma(\sigma)= \deg_\Gamma(\sigma/\gen{}{}) =\deg\big(Q(\Gamma, \sigma;\, x) \big)=I(\Gamma, \sigma).$

 		\item $\delta_\Gamma(\sigma'/\sigma) = \frac{\,\deg_\Gamma(\sigma'/\sigma)\,}{\deg(\Gamma)}$.\\  It's worth noting that $\deg(\Gamma) = \max \{\deg_\Gamma(\sigma): {\sigma \in \Gamma}\}$ and $\deg_\Gamma(\sigma'/\sigma)=\deg_\Gamma(\sigma')-\deg_\Gamma(\sigma)$. 
 		
 		\item If $\delta_\Gamma(\sigma'/\sigma) =  \deg_\Gamma(\sigma'/\sigma) = 0$, then $Q(\Gamma,\sigma,\sigma';\,x)$ is constant one and we set  $\mu_\Gamma(\sigma'/\sigma) = Q(\Gamma,\sigma,\sigma';\,x)=1$. Otherwise, we set
 		$$\mu_\Gamma(\sigma'/\sigma) = \lim_{x\to\infty}\frac{Q(\Gamma, \sigma, \sigma' ;\,x)}{\,\,P(\Gamma;\,x)^{\delta_\Gamma(\sigma'/\sigma)}\,\,} = \frac{1}{\,A(\Gamma)^{\delta_\Gamma(\sigma'/\sigma)}\,}.$$
 		where $A(\Gamma)$ is the coefficient of the maximum degree monomial $P(\Gamma;\,x)$.  Note that $A(\Gamma)$ can also be described as the number of nodes in $\Gamma$ of maximal degree.
 	\end{itemize}
 	We will usually omit the subscript $\Gamma$ from the notation, as the tree plan $(\Gamma,\lambda)$ will be clear from context.
 \end{defn}
 
 \begin{lemma}
 	Let $\Gamma = (\Gamma,\lambda)$ be a tree plan, and let $\sigma\subseteq\sigma'\in\Gamma$. Then, for every finite set $X$ and for every $b\in \Gamma(X)$ such that $\pi_X(b) = \sigma$, 
 	$$Q(\Gamma,\sigma,\sigma';\,|X|) = \big|\Gamma(X; b,\sigma')\big|,$$
 	where
 	$$\Gamma(X; b,\sigma') = \left\{a\in\Gamma(X)\Bvert \pi_X(a) = \sigma'\textit{ and }b<a\right\}.$$
 \end{lemma}
 \begin{proof}
 	Let $\sigma'=\sigma\concat\tau$. Then by Lemma \ref{Q(..,|X|)}
 	we have 	$$Q(\Gamma, \sigma, \sigma';\, |X|)=Q(\Gamma_\sigma,\tau;\,|X|)=|\left\{ a\in \Gamma_\sigma(X) : \pi_X(a)=\tau \right\}|.$$
 	\begin{claim}  There is a bijection between 
 		$\left\{ a\in \Gamma_\sigma(X) : \pi_X(a)=\tau \right\} $ and\\ $ \left\{ a\in \Gamma(X) : \pi_X(a)=\sigma'  \wedge b<a \right\}$.
 	\end{claim}
 	\begin{proof} 
 		The map $a\mapsto b\concat a$ defines a bijection between them.  Let $a$ be a member of the left hand set. Then $b< b\concat a $ and $\pi_X(b\concat a)=\sigma\concat\tau =\sigma'$. So $b\concat a$ belongs to the right hand set. If $a, a'$ are distinct members of  $\Gamma_\sigma(X)$, then $b\concat a$ and $b\concat a'$ are also distinct. Also, for  an arbitrary member $c$ of the right hand set,  since $b<c$ and $\pi_X(b)=\sigma$,  it can be written as $b\concat a$ for some member $a\in \Gamma_\sigma(X)$. Since $\pi_X(c)= \pi_X(b\concat a)=\sigma'=\sigma\concat\tau$  and $\pi_X(b)=\sigma$, $\pi_X(a)=\tau$. 
 	\end{proof}
 	This completes the proof.   
 \end{proof}
 
 
 \begin{prop}\label{mu}
 	Let $\Gamma = (\Gamma,\lambda)$ be a tree plan, and let $\sigma\subseteq\sigma'\in\Gamma$ with $\deg(\sigma'/\sigma)>0$. Let $X_0$ be a finite set, and let $b\in \Gamma(X_0)$ such that $\pi_{X_0}(b) = \sigma$. Then 
 	$$\lim_{|X|\to\infty}\frac{\,\big|\Gamma(X;b,\sigma')\big|\,}{\big|\Gamma(X)\big|^{\delta_\Gamma(\sigma'/\sigma)}} = \mu_\Gamma(\sigma'/\sigma)$$
 	where $X$ ranges over finite sets containing $X_0$.  
 \end{prop}
 \begin{proof}
 	By definition of $\mu_\Gamma(\sigma'/\sigma)$ we have
 	\begin{align*}
 	\mu_\Gamma(\sigma'/\sigma)  &=  \lim_{x\to\infty}\frac{Q(\Gamma,\sigma, \sigma';\, x)}{P(\Gamma;\, x)^{\delta_\Gamma(\sigma'/\sigma)}}\\ &= \lim_{|X|\to\infty}\frac{Q(\Gamma, \sigma, \sigma';\, |X|)}{P(\Gamma;\, |X|)^{\delta_\Gamma(\sigma'/\sigma)}}\\ &= \lim_{|X|\to\infty}\frac{|\Gamma(X;\, b, \sigma')|}{|\Gamma(X)|^{\delta_\Gamma(\sigma'/\sigma)}},
 	\end{align*}
 	where $X$ ranges over finite sets containing $X_0$.  
 \end{proof} 
 
 
 \subsection{Asymptotic Estimates for Complete 1-types}
 
 It will be convenient to define, for $B\subset_\fin \Gamma(\omega)$ and $e\in\tcl(B)$, 
 $$\mult(e/B) = \big|\big\{g(e):g\in Aut(\Gamma(\omega)/B)\big\}\big|,$$
 the \textit{multiplicity} of $e$ over $B$ which is always 1. In fact, any automorphism fixing $B$ fixes all of $\tcl(B)$.

 \begin{prop}\label{2.3.2.1} 
 	Let $\bb$ be an enumeration of $B\subset_\fin\Gamma(\omega)$, and let $a\in\Gamma(\omega)$. Let $\theta(x,\bb)$ be a principal formula of $\tp_\Gamma(a/B)$. 
 	\begin{enumerate}
 		
 		\item Suppose $a\in\tcl(B)$. Let $b$ be the maximum element of ${\downarrow}B$ such that $b\leq a$, and let $\sigma = \pi(b)$ and $\sigma' = \pi(a)$. Then 
 		$$\big|\theta(\Gamma(X),\bb)\big| = \mult(a/B) = \mu(\sigma'/\sigma)=1$$
 		for all sufficiently large sets $X\subset_\fin S$ such that $B\subseteq\Gamma(X)$.
 		
 		\item Suppose $a\notin\tcl(B)$. Let $e = [a\wedge B]$, $\sigma= \pi(e)$, and $\sigma' = \pi(a)$. Then 
 		$$\lim_{|X|\to\infty}\frac{\,\big|\theta(\Gamma(X),\bb)\big|\,}{\,\,|\Gamma(X)|^{\delta(\sigma'/\sigma)}\,\,} = \mult(e/B){\cdot}\mu(\sigma'/\sigma)=\mu(\sigma'/\sigma)$$
 		where $X$ ranges over finite subsets of $\omega$ such that  $B\subseteq\Gamma(X)$.
 	\end{enumerate}
 \end{prop}
 
 \begin{proof}  
 	\begin{enumerate}
 		\item Let $a\in \tcl(B)$, and let $b$ be the maximum element of ${\downarrow}B$ such that $b\leq a$, and let $\sigma = \pi(b)$ and $\sigma' = \pi(a)$. We consider the following cases.
 		\begin{itemize}
 			
 			\item[case 1.1]
 			If $a\in B$, then $\theta(x,\bb)=(x=a)$ isolates $\tp_\Gamma(a/B)$. So $\theta(\Gamma(X), \bb)=\{a\}$. Then $|\theta(\Gamma(X), \bb)|=1$. Since $a\in B$, $g(a)=a$ for each $g\in Aut(\Gamma(\omega)/B)$. This implies $\mult(a/B)=1$. Also since $a\in B$, $b=a$ and $\sigma=\sigma'$. So we have
 			$$ Q(\Gamma, \sigma, \sigma';\, x)=Q(\Gamma_\sigma, \gen{}{};\, x)=1.$$
 			Hence $\mu(\sigma'/\sigma)=Q(\Gamma, \sigma, \sigma';\, x)=1$. 
 			
 			\item[case 1.2]
 			If $a\in{\downarrow}B\setminus B$, then there is   $c\in B$ such that $\pred^k(c)=a$ for some $k$. So $\theta(x,\bb)=(\pred^k(c)=a)$ is the formula that isolates $\tp_\Gamma(a/B)$. So $\theta(\Gamma(X), \bb)=\{a\}$. Then $|\theta(\Gamma(X), \bb)|=1$.  Since every automorphism $g\in Aut(\Gamma(\omega)/B)$ fixes ${\downarrow}B$ as well,   $\mult(a/B)=1$.  Also since $a\in {\downarrow}B\setminus  B$, $b=a$ and $\sigma=\sigma'$. So 
 			similar to the previous case we have $\mu(\sigma'/\sigma)=Q(\Gamma, \sigma, \sigma';\, x)=1$.

 			\item[case 1.3]
 			If $a\in \tcl(B)\setminus {\downarrow}B$, then $a\notin I(S)$. Let $c_1, \dots, c_t\in \Gamma(\omega)$ be an enumeration of the successors of $\pred(a)$. There is $k\in \NN^+$ such that $\pred^k(a)=b$ (note that in this case $b\in B$). Let $\phi(x)$ be the formula which describes the sub-tree rooted at $a$ in $\Gamma(X)$ for some sufficiently large set $X\subset_\fin S$ such that  $B\subseteq\Gamma(X)$. Then $\theta(x,b)=$
 			\begin{align*}
 			&	\pred^k(x) = b \wedge P_\sigma(x)
 			\end{align*}
 			isolates $\tp(a/B)$. Since there is no infinity point in the path from $b$ to $a$, $Q(\Gamma, \sigma, \sigma';\, x)=\mult(a/B)$. Also since $Q(\Gamma, \sigma, \sigma':\, x)$ is constant, $\mu(\sigma'/\sigma)=Q(\Gamma, \sigma, \sigma';\, x)$.  
 			
 		\end{itemize}
 		\item Let $a\notin\tcl(B)$. Let $e = [a\wedge B]$, $\sigma= \pi(e)$, and $\sigma' = \pi(a)$. By Lemma \ref{mu} we have
 		$$\lim_{|X|\to\infty}\frac{\,\big|\Gamma(X;b,\sigma')\big|\,}{\big|\Gamma(X)\big|^{\delta_\Gamma(\sigma'/\sigma)}} = \mu_\Gamma(\sigma'/\sigma).$$
 		So it is enough to show that for every sufficiently large finite set $X$,
 		$$  |\theta(\Gamma(X), \bb)|=\mult(e/B)\cdot |\Gamma(X; e, \sigma')|= |\Gamma(X; e, \sigma')|.$$  	\end{enumerate}
 \end{proof}

 
 \subsection{Lifting Asymptotic Estimates to Arbitrary Formulas}
 
 \begin{thm}\label{deg(Gamma)-dim-asymptotic-class}
 	$\KK(\Gamma)$ is a $\deg(\Gamma)$-dimensional asymptotic class, and it is not an $n$-dimensional asymptotic class for any $n<\deg(\Gamma)$.
 \end{thm}
 
 \begin{proof}
 	Let $\Gamma=(\Gamma, \lambda)$ be a tree plan with $\deg(\Gamma)=N$.
 	Suppose $\phi(x;y_0,\dots, y_{m-1})$ is an arbitrary $\LL_t$-formula.   Let $\aa_0, \dots, \aa_{r-1}\in \Gamma(\omega)^m$  such that $$\big(\forall \bb\in \Gamma(\omega)^m \big)\big(\exists ! \,\,i<r\big)\, \tp(\bb)=\tp(\aa_i). $$
 	For each $i<r$, let $\theta_i(\yy)$ be a quantifier-free formula which isolates $\tp(\aa_i)$. Then $\theta_0(\yy), \dots, \theta_{r-1}(\yy)$ partition $\big\{(A, \bb) : A\in\TT(\Gamma)\text{ and } \bb\in A^m\big\}$: i.e.,
 	\begin{itemize}
 		\item $Th(\Gamma)\models\bigvee_{i<r}\theta_i(\yy)$, and
 		\item $Th(\Gamma)\models \theta_i(\yy)\to \neg\theta_j(\yy)\,\,$ for $i\neq j$. 
 	\end{itemize}

 	For each $i<m$, let $b_{i,0}, \dots, b_{i, n_{i-1}}\in \Gamma(\omega)$ such that 
 	$$ \big( \forall b\in \phi(\Gamma(\omega), \aa_i)\big) \big( \exists ! \,\, j<n_i\big)\Big( \tp(b, \aa_i)  =  \tp(b_{i,j}, \aa_i) \Big)$$	
 	For each $i<r$ and each $j<n_i$, let $\psi_{i,j}(x,\yy)$ be a quantifier-free formula which isolates $\tp(b_{i,j}, \aa_i)$. Then $\phi(x,\yy)$ is equivalent in $Th(\Gamma)$ to
 	$$ \bigwedge\limits_{i<r}\Big(\theta_i(\yy)\to \bigvee\limits_{j<n_i} \psi_{i,j}(x,\yy) \Big) $$
 	where $\psi_{i,j}(x,\yy) = \qtp(b_{i,j}, \aa_i)$  and  $\pi(b_{i,j})=\sigma'_{i,j}\supseteq \sigma_{i,j}=\pi\big([b_{i,j}\wedge \tcl(\aa_i)]\big)$.   
 	
 	Case 1: if $\deg(\sigma'_{i,j}/\sigma_{i,j})=0$, then
 	\begin{align*}
 	\big\vert \vert\psi_{i,j}(\Gamma(X),b_{i,j})\vert -&\mu(\sigma'_{i,j}/\sigma_{i,j})\vert \Gamma(X)\vert^{\delta_\Gamma(\sigma'_{i,j}/\sigma_{i,j})}\big\vert \\ &=\big\vert 1-1\cdot \vert \Gamma(X)\vert^0\big\vert \\ &=0.
 	\end{align*}
 	Case 2: if $\deg(\sigma'_{i,j}/\sigma_{i,j})>0$, by Proposition \ref{mu}, for each $i<r$ and $j<n_i$, we have
 	$$\lim_{|X|\to\infty}\frac{\,\big|\psi_{i, j}(\Gamma(X),b_{i,j})\big|\,}{\big|\Gamma(X)\big|^{\delta_\Gamma(\sigma'_{i,j}/\sigma_{i,j})}} = \mu_\Gamma(\sigma'_{i,j}/\sigma_{i,j}).$$
 	It follows that 	$$\lim_{|X|\to\infty}\frac{\,\Big|\big|\psi_{i,j}(\Gamma(X),b_{i,j})\big| - \mu_\Gamma(\sigma'_{i,j}/\sigma_{i,j}) \big| \Gamma(X)\big|^{\delta_\Gamma(\sigma'_{i,j}/\sigma_{i,j})} \Big|\,}{\big|\Gamma(X)\big|^{\delta_\Gamma(\sigma'_{i,j}/\sigma_{i,j})}} = 0$$
 	Thus $\Big|\big|\psi_{i,j}(\Gamma(X),b_{i,j})\big| - \mu_\Gamma(\sigma'_{i,j}/\sigma_{i,j}) \big| \Gamma(X)\big|^{\delta_\Gamma(\sigma'_{i,j}/\sigma_{i,j})} \Big|=o\Big( \big|\Gamma(X)\big|^{\delta_\Gamma(\sigma'_{i,j}/\sigma_{i,j})}\Big)$.
 	
 	Now, let $(A, \cc)\in \big\{(A, \bb) : A\in\TT(\Gamma)\text{ and } \bb\in A^m\big\} $ 
 	There is $i=i(\cc)<r$ such that $A\models \theta_i(\cc)$. Then $\phi(A,\cc)= \bigcup_{j<n_i} \psi_{i,j}(A,\cc)$. Let $\delta(\phi, \cc)=\max\big\{ \delta_\Gamma(\sigma'_{i,j}/\sigma_{i,j})  : j<n_i \big\}$. Then
 	\begin{align*}
 	\big|\phi(A,\cc)\big| & = \sum_{j<n_i }\big|\psi_{i,j}(A,\cc)\big|  
 	\end{align*}
 	So by Proposition \ref{2.3.2.1}:
 	\begin{align*}
 	\lim\limits_{\vert A\vert\to\infty} \frac{\vert \phi(A,\cc) \vert}{\vert A \vert^{\delta(\phi,\cc)}} &=\lim\limits_{\vert A\vert\to\infty}\sum\limits_{j<n_i}\frac{\mu(\sigma'_{i,j}/\sigma_{i,j})}{\vert A\vert^{\delta(\phi,\cc)-\delta(\sigma'_{i,j}/\sigma_{i,j})} }\\
 	&= \sum\limits_{j<n_i, \,\,\, \delta(\sigma'_{i,j}/\sigma_{i,j})=\delta(\phi, \cc)} \mu(\sigma'_{i,j}/\sigma_{i,j}).
 	\end{align*}
 
 	It follows that $$\Big|\big|\phi(A, \cc)\big| - \mu(\phi, \cc) \big| A\big|^{\delta(\phi, \cc)} \Big|=o\Big( \big| A\big|^{\delta(\phi, \cc)} \Big).$$
 	Note that the values of $\delta(\phi, \cc)$ and $\mu(\phi, \cc)$ only depend on $\phi$ and $\tp(\cc/\emptyset)$, so only finitely many values appear for each $\phi$, and the $\cc$ which give those values are definable.
 \end{proof}


 \section{Characterization of  $\aleph_0$-categorical Trees}\label{Characterization of  aleph0-categorical Trees} 
 In this section, we will show that the $\aleph_0$-categorical theories of trees can be obtained from tree plans.
 
 \begin{thm}[The Characterization Theorem for $\aleph_0$-categorical Trees]\label{CharOmega-Cat}
 	Let $T$ be either an $\aleph_0$-categorical theory of infinite trees or  a theory of finite trees (as $\LL_t$-structures). Then there is a tree plan $\Gamma$ such that $T = \Th(\Gamma(\omega)\r\LL_t)$.
 \end{thm}
 \begin{proof}
 	Let $\M\models T$ and $|M|=\aleph_0$.	First, note that $\M$ has finite height. This is clear in the finite case. 
 	In the $\aleph_0$-categorical case, the formulas $\pred^k(x) = \epsilon$ cannot be pairwise inequivalent by Ryll-Nardzewski's  Theorem. We argue by induction on the height of $\M$.
 	If $\M$ has height 0, $\M = {\epsilon}$, so we can take $\Gamma$ to be the unique tree plan on $\{\gen{}{}\}$ which assigns 1 to the root. 
 	Suppose $\M$ has height $n+1$. Let $E$ be the set of successors of the root. For each $e$ in $E$, let $M_e={\uparrow}e$ be the subtree above $e$. Now there are only finitely many isomorphism classes in $\{\M_e : e\in E\}$. This is clear in the finite case, and it follows from the  Ryll-Nardzewski Theorem in the $\aleph_0$-categorical case. There are only finitely many types of elements in $E$, and if $\tp(e) = \tp(e')$, there is an automorphism of $\M$ moving $e$ to $e'$, which restricts to an isomorphism $\M_e$ to $\M_{e'}$. Take representatives $\M_{e_1},\dots,\M_{e_n}$ for the isomorphism classes. Each $\M_{e_i}$ is finite or $\aleph_0$-categorical, since it is definable in $\M$. And $\M_{e_i}$ has height $n$, so by induction hypothesis there is a tree plan $\Gamma_i$ such that $\M_{e_i}$ is isomorphic to $\Gamma_i(\omega)$. To build the tree plan $\Gamma$: For each $i$ between $1$ and $n$, if the isomorphism type of $\M_{e_i}$ appears finitely many times (say $k$ times) in $\{M_e : e \in E\}$, attach $k$ copies of $\Gamma_i$ to the root in $\Gamma$, and label each of their roots by $1$. If the isomorphism type of $\M_{e_i}$ appears infinitely many times in  $\{\M_e : e \in E\}$, attach one copy of $\Gamma_i$ to the root in $\Gamma,$ but make the root of this copy of $\Gamma_i$ an infinity node. It follows almost immediately by induction that $\M$ is isomorphic to $\Gamma(\omega)$.
 \end{proof}
 


 \section{Characterization of Asymptotic Classes of Trees}\label{Characterization of Asymptotic Classes of Trees}

  In this section, we will provide a characterization of supersimple trees with finite rank that eliminate $\exists^\infty$. We also show that asymptotic classes of finite trees yield $\aleph_0$-categorical ultraproducts.  Also we provide a characterization of  asymptotic classes of finite trees using the notion of tree plan. 
 
 	
 

 Throughout this section, let $\KK$ be an 
  asymptotic class (see Definition \ref{1-dim}) of finite trees with arbitrarily large members. We recall the following notation
 $$\mathbf{C}_\KK =\big\{ \M: \M \textrm{ is an infinite ultraproduct of members of } \KK  \big\},$$
 $$Th(\mathbf{C}_\KK) = \bigcap\big\{Th(\M):\M\in \mathbf{C}_\KK \big\},$$ and $$\TT_{\mathbf{C}_\KK} =\big\{Th(\M):\M\in \mathbf{C}_\KK\big\}.$$ 
 That is, $Th(\mathbf{C}_\KK)$ is the common theory of all infinite ultraproducts of members of $\KK$, and $\TT_{\mathbf{C}_\KK}$ is the set of completions of $Th(\mathbf{C}_\KK)$. So $Th(\mathbf{C}_\KK)=\bigcap\TT_{\mathbf{C}_\KK}$.

 \begin{fact}\label{fact:finite-height}
 If $\KK$ is an 
 asymptotic class of finite trees with arbitrary large members, then there is an upper bound for the height of members of $\KK$. Notice, towards a contradiction, that if for each $n<\omega$, there is a tree $A_n\in \KK$ with height at least $n$, then the formula $x<y$ witnesses that any 
 infinite ultraproduct of $A_i's$ has SOP which is a contradiction because it is supersimple, and so it is NSOP.
 \end{fact}

Recall that a theory $T$ eliminates the quantifier $\exists^\infty$ if for every formula $\phi(x,\yy) $,  the set of tuples $\bb$ in a model $\mathcal{M} \models T$ such that $\phi(\M,\bb)$ is infinite, is definable. Equivalently, by compactness, there exists a positive integer $N = N(\phi)$ such that for every model $\mathcal{\M} \models T$ and tuple $\bar{b}\in M^{|\yy|}$, the set $\phi(\M,\bb)$ is either finite of size less than or equal to $N$, or infinite.

\begin{thm}\label{characterization of supersimple trees}
Let $\M$ be an infinite tree (as $\LL_t$-structure). The following statements are equivalent.
\begin{enumerate}
\item $\M\equiv \Gamma(\omega)$ for some tree plan $\Gamma$.
\item $\M$ is elementarily equivalent to an infinite ultraproduct of an N-dimensional asymptotic class of finite trees.
\item $Th(\M)$ is supersimple, finite rank and  eliminates $\exists^\infty$.
\item $Th(\M)$ has finite height and eliminates $\exists^\infty$.
\item $Th(\M)$ is $\aleph_0$-categorical. 
\end{enumerate}
\end{thm}
\begin{proof}
By Theorem \ref{CharOmega-Cat}, it's immediate that $(1)\Leftrightarrow (5)$. For (1) $\Rightarrow$ (2), it's enough to consider the asymptotic class $\KK(\Gamma)$ (see Definition \ref{Def-K(Gamma)}). By Fact \ref{supersimple finite su-rank}, it is clear that (2) $\Rightarrow$ (3) $\Rightarrow$ (4).
To complete the proof it is enough to show that $(4) \Rightarrow (1)$.  Let $\M$ be a countably infinite tree  of finite height that eliminates $\exists^\infty$. We do induction on the height of $\M$ similar to the proof of \ref{CharOmega-Cat}.  Note that because of elimination of $\exists^\infty$ and finiteness of the height,  only finitely many tree plans (up to isomorphism) are possible for subtrees of height $\leq n$. By gluing these tree plans we get our desire tree plan $\Gamma$. More precisely, 
if $\h(\M)=0$, then $\M=\epsilon$. So we let  $\Gamma$  to be the tree plan $\gen{}{}$ which assigns 1 to the root. Suppose the statement is true for every tree of height $\leq n$. Let $\h(\M)=n+1$. Let $E$ be the set of successors of the root. For each $e\in E$, let $\M_e$ be the subtree above $e$. If $E$ is finite then  $E=\{e_0, \dots, e_{k-1}\}$ for some $k<\omega$. For each $i<k$, the subtree $\M_{e_i}$ has height $n$, and so by induction hypothesis there is a tree plan $\Gamma_i$ such that $\M_{e_i}\cong \Gamma_i(\omega)$. Let $\Gamma_0=\{\gen{}{}\}\cup \{\gen{i}{} : i<k\}$ be a tree plan such that $\lambda(\gen{}{})=\lambda(\gen{i}{})=1$ for each $i<k$. To build the tree plan $\Gamma$, we start with $\Gamma_0$ and glue $\Gamma_{e_i}$ to $\gen{i}{}$ for each $i<k$. For the case that $E$ is an infinite set, because of elimination of $\exists^\infty$ we have finitely many possibilities, up to isomorphism, for subtrees above the members of $E$. By the induction hypothesis there exists a tree plan for each $\M_e$. And if we continue this process for successor of members of $E$, this process will stop because  the height of $\M$ is finite. Therefore we can build $\Gamma$ by gluing finitely many tree plans.
\end{proof} 

The following Corollaries are immediate from Theorem \ref{CharOmega-Cat} and Theorem \ref{characterization of supersimple trees}.

 \begin{cor}\label{Prop-finiteness+OmegaCat} 
 	Let $\KK$ be a $N$-dimensional asymptotic class of finite trees with arbitrarily large members. 
 	Then 
 	\begin{enumerate}
 		\item $\TT_{\mathbf{C}_\KK}$ is finite.
 		\item Every $T\in \TT_{\mathbf{C}_\KK}$ is $\aleph_0$-categorical.
 	\end{enumerate}
 \end{cor}

 \begin{cor}[Characterization of Asymptotic Classes of Trees]\label{Char1-dimAsymp} 
 Let $\KK$ be an isomorphism closed class of finite trees (as $\LL_t$-structures).  $\KK$ is an asymptotic class, if and only if there are finitely many tree plans $(\Gamma_0, \lambda_0), \dots, (\Gamma_{k-1}, \lambda_{k-1})$ such that 
 		$$\TT_{\mathbf{C}_\KK}=\bigg\{ Th\big(\Gamma_i(\omega)\big) : i<k \bigg\}. $$
 \end{cor}



 \section{More on the Model Theory of $\Gamma(\omega)$}\label{model theoretic properties of Gamma(omega)}
 In this section we study basic model theoretic properties of the generic model $\Gamma(\omega)$ including types and forking behavior. 
 \subsection{Types in $\Gamma(\omega)$}

 \begin{conv}
 	For brevity, we write: $\tcl(B)$ in place of $\tcl^{\Gamma(\omega)}(B)$, 
 	$\acl(B)$ in place of $\acl^{\Gamma(\omega)}(B)$, and 
 	$\tp_\Gamma(...)$ when we mean $\tp^{\Gamma(\omega)}(...)$.
 	
 \end{conv}

 \begin{lemma}\label{tcl=acl} 
 	$\acl = \tcl$.
 \end{lemma}
 \begin{proof}
 	Let $B\subset \Gamma(\omega)$. First, by induction on $n$ we show that $\bigcup_n \tcl_n(B)=\tcl(B)\subseteq \acl(B)$. 	
 	\begin{itemize}
 		\item If $n=0$, then we show that $\tcl_0(B)\subseteq \acl(B)$. Let $a\in \tcl_0(B)$, then $a\leq b$ for some $b\in B$. There is a number $k<\omega$ such that $\pred^k(b)=a$. So the formula $\pred^k(b)=x$ is an algebraic formula witnesses that $a\in \acl(B)$. 
 		
 		\item	Suppose for all $\ell\leq n$, $\tcl_\ell(B)\subseteq \acl(B)$. Let  $a\in \tcl_{n+1}(B)$, then $\pred(a)\in \tcl_n(B)$ and $\lambda(\pi(a))=1$. Get $b\in \tcl_n(B)$ such that $\pred(a)=b$. Since $b\in \acl(B)$, the formula $\pred(x)=b \wedge P_{\pi(a)}(x)$ witnesses that $a\in \acl(B)$.
 	\end{itemize}
 	Now we show that $\acl(B)\subseteq \tcl(B)$. Let $a\notin \tcl(B)$ and $e=[a\wedge \tcl(B)]$. If there is no  infinity point between $e$ and $a$, then $a\in \tcl(B)$. So there must be an infinity point, say $b$, between $e$ and $a$.  So there are infinitely many copies of $b$ with the same types over $B$.  It follows that 
 	$$\big| \big\{ g(a)  :  g\in Aut\big(\Gamma(\omega)/B \big) \big\} \big|=\infty $$
 	and hence $a\notin \acl(B)$. 
 \end{proof}
 
 \begin{fact}\label{pi_determines_tp}
 	Let $a,a',b,b'\in\Gamma(\omega)$. 
 	\begin{enumerate}
 		\item If $\pi(a) = \pi(a')$, then $\tp_\Gamma(a) = \tp_\Gamma(a')$.
 		\item If $a<b$, $\pi(a) = \pi(a')$, and $\pi(b)=\pi(b')$, then $\tp_\Gamma(a,b) = \tp_\Gamma(a',b')$ 
 	\end{enumerate}
 \end{fact}

 \begin{lemma}\label{a/e_imples_a/B} 
 	Let $B\subset_\fin\Gamma(\omega)$ and $a\in\Gamma(\omega)$. If $a\notin\tcl(B)$ and $e=[a\wedge B]$, then $\tp_\Gamma(a/e)\models \tp_\Gamma(a/B)$, meaning that for any $a'\in\Gamma(\omega)$, if $a'\models\tp_\Gamma(a/e)$, then  $a'\models\tp_\Gamma(a/B)$. 
 \end{lemma}
 \begin{proof}
 	Let $B\subset_\fin\Gamma(\omega)$ and $a\in\Gamma(\omega)$ be given, and  $e=[a\wedge B]$. Let $P$ be the path from $e$ to $a$. By induction on the length of $P$ we show that $\tp_\Gamma(a/e)\models \tp_\Gamma(a/B)$. 
 	\begin{itemize} 
 		\item 	If length of $P$ is equal to 1, then $\pred(a)=e$ and $a\in I(S)$. Also if $\pred(a')=e=\pred(a)$ and $\pi(a)=\pi(a')$, then by Fact \ref{pi_determines_tp} we have $\tp_\Gamma(a)=\tp_\Gamma(a')$. Also since $e$ is definable from $a$ and definable from $a'$, $\tp_\Gamma(ae)=\tp_\Gamma(a'e)$. So among the members of the set $\{b : \pred(b)=e\}$, the function $\pi$ determines what the orbit is. This means that when we fix $B$, the set $D= \{b : \pred(b)=e=\pred(a)\,\, \wedge\,\, \pi(a)=\pi(b)\}$  consists of the realizations of $\tp_\Gamma(a/e)$ which forms an orbit over $B$ as well. In other words, given any two members $b, b'$ of $D$, we can make a embedding that fixes $\tcl(B)$ and takes $b$ to $b'$.  Such a embedding can be extended to an automorphism. So every realization of $\tp_\Gamma(a/e)$ realizes $\tp_\Gamma(a/B)$. 
 		
 		\item Suppose for the case that the length of $P$ is $\leq n$, the statement holds.
 		
 		\item  Let $P$ have length $n+1$. Let $a'=\pred(a)$, and let $P'$ be the  path from $e=[a\wedge B]=[a'\wedge B]$ to $a'$.  $P'$ has length $\leq n$ and so by induction hypothesis we have $\tp_\Gamma(a'/e)\models \tp_\Gamma(a'/B)$. 
 		\begin{claim}
 			$\tp_\Gamma(a/a')\models \tp_\Gamma(a/a' B)$
 		\end{claim}
 		\begin{proof}
 			The proof is the same as the case when length of $P$ is 1, using $a'B$ in place of $B$.
 		\end{proof}	
 		Note that	$\tp(a'/e)\models \tp(a'/B)$ and $\tp(a/a')\models \tp(a/a'B)$ together imply $\tp(a/e)\models \tp(a/B)$.
 	\end{itemize}
 \end{proof}
 

 \subsection{Forking in Models of $\Th(\Gamma)$}
 
 Let $\Gamma =(\Gamma,\lambda)$ be a  tree plan. We use the same notational conventions as in the previous section. 
 
 \begin{prop}
 	Let $C\subseteq B\subset_\fin \Gamma(\omega)$ and $a\in\Gamma(\omega)$.
 	\begin{enumerate}
 		\item If $a\in\tcl(C)$, then $\tp_\Gamma(a/B)$ does not divide over $C$.
 		\item Assume $a\notin\tcl(C)$. Then the following are equivalent:
 		\begin{enumerate}
 			\item $\tp_\Gamma(a/B)$ divides over $C$
 			\item There is $e\in \big(I(S)\cap\tcl(B)\big)\setminus \tcl(C)$  such that $[a\wedge C]<e\leq  a$.  
 		\end{enumerate}
 	\end{enumerate}
 \end{prop}
 
 \begin{proof}
 	\begin{enumerate}
 		\item Let $a\in \tcl(C)$.   
 		Let $\phi(x,\bb)$ be a formula in $\tp_\Gamma(a/B)$ where $\bb$ is a tuple from $B$. If $\phi(x, \bb)$ divides over $C$, then there is a $C$-indiscernible sequence $(\bb_i)_{i<\omega}$ with $\bb_0=\bb$ such that $\big\{\phi(x,\bb_i) : i<\omega \big\}$ is $k$-inconsistent for some $k<\omega$.  By Lemma \ref{tcl=acl} we have  $a\in \acl(C)$. So there is only  finitely many copies of $a$ over $C$ available.  Let $\{a_1, \dots, a_n\}$ be the conjugates of $a$ over $C$ with $a_1=a$. Since $(\bb_i)_{i<\omega}$ is a $C$-indiscernible, for every $i<\omega$ there is $1\leq j_i\leq n$ such that $\models \phi(a_{j_i},\bb_i)$. By the pigeonhole principle there are $1\leq j\leq n$ and an infinite set $I\subseteq \omega$ such that $\models \phi(a_j, \bb_i)$ for all $i\in I$. This contradicts $k$-inconsistency. 
 		
 		\item $(a)\Rightarrow (b)$ Let $e=[a\wedge B]$. By Lemma \ref{a/e_imples_a/B} we have $\tp(a/e)\models\tp(a/B)$. 
 		
 		Let $\phi(x,e)$ 
 		be a formula that isolates $\tp(a/e)$ and let $\psi(x,\bb)$ be a formula in $\tp(a/B)$ which divides over $C$. 
 		Since $\phi(x,e)\models \psi(x,\bb)$ and $\psi(x,\bb)$ divides over $C$, $\phi(x,e)$ divides over $C$ as well. So 
 		there is a $C$-indiscernible sequence $(e_i)_{i<\omega}$ with $e_0=e$ such that $\{\phi(x,e_i) : i<\omega\}$ is $k$-inconsistent for some $k<\omega$.  Since $e$ belongs to a non-constant indiscernible sequence over $C$, $e\in I(S)$.
 		
 		
 		
 		$(b)\Rightarrow (a)$ Let $E$ be  the set of all conjugates of $e$ over $C$.  
 		\begin{claim}
 			$\tp_\Gamma(a/e)$ divides over $C$. 
 		\end{claim}
 		\begin{proof}
 			Let $k$ be the smallest positive integer such that $\pred^k(a)=e$. We will show that the formula  $\phi(x,e)=\big(\pred^k(x)=e \big)\in \tp_\Gamma(a/e)$ divides over $C$. Let $(e_i)_{i<\omega}$ be an enumeration of $E$. Then $(e_i)_{i<\omega}$ is a $C$-indiscernible sequence and $\{\pred^k(x)=e_i : i<\omega\}$ is 2-inconsistent. So $\tp(a/e)$ divides over $C$.
 		\end{proof}
 		
 		\begin{claim}
 			$e\in B$.
 		\end{claim}
 		\begin{proof}
 			If $e\notin B$, then since $e\in I(S)$, $e\notin \tcl(B)$ which is a contradiction.  
 		\end{proof}
 		The  previous  two claims together imply that $\tp_\Gamma(a/B)$ divides over $C$. 
 	\end{enumerate}
\end{proof}
 
\section*{Acknowledgment}
The author would like to thank his PhD advisor, Cameron Hill, for her very helpful guidance and assistance.  The author also would like to thank Alex Kruckman for his reading of the first draft and for the helpful comments.

 	\bibliography{Asymptotic-Trees_Mirabi}
 \bibliographystyle{abbrv}
 
\end{document}